\documentclass[a4paper,11pt]{amsart}

\usepackage{a4wide}
\usepackage[T1]{fontenc}
\usepackage{lmodern}
\usepackage{amsmath,amssymb,amsthm,mathtools}
\usepackage{enumitem}
\usepackage{microtype}
\usepackage{float}
\usepackage{tikz-cd}
\usepackage{url}
\usepackage[colorlinks=true,linkcolor=black,citecolor=black,urlcolor=blue]{hyperref}

\hypersetup{
  pdftitle={Grothendieck and ell-infinity-Grothendieck subspaces of ell-infinity},
  pdfauthor={Manuel Gonzalez and Tomasz Kania},
  pdfsubject={Grothendieck subspaces of ell-infinity},
  pdfkeywords={Grothendieck space, ell-infinity-Grothendieck subspace, reflexive quotient, Fredholm operator, stable random variable}
}

\allowdisplaybreaks
\setlist[enumerate]{label=\textup{(\roman*)},leftmargin=2.2em}

\newcommand{\K}{\mathbb K}
\newcommand{\N}{\mathbb N}
\newcommand{\R}{\mathbb R}
\newcommand{\C}{\mathbb C}
\newcommand{\cF}{\mathfrak c}
\newcommand{\ran}{\operatorname{ran}}
\newcommand{\Nstar}{\N^*}
\newcommand{\ontoarrow}{\twoheadrightarrow}

\newtheorem{theorem}{Theorem}[section]
\newtheorem{proposition}[theorem]{Proposition}

\newtheorem{corollary}[theorem]{Corollary}
\newtheorem{maintheorem}{Theorem}

\theoremstyle{definition}
\newtheorem{definition}[theorem]{Definition}
\newtheorem{example}[theorem]{Example}
\theoremstyle{remark}

\title[Grothendieck subspaces of $\ell_\infty$]
{Grothendieck and $\ell_\infty$-Grothendieck subspaces of $\ell_\infty$}

\author[M.~Gonz\'alez]{Manuel Gonz\'alez}
\address[M.~Gonz\'alez]{Departamento de Matem\'aticas, Estad\'istica y Computaci\'on\\
Facultad de Ciencias\\
Universidad de Cantabria\\
Avenida de los Castros, 48\\
39005 Santander, Spain}
\email{manuel.gonzalez@unican.es}

\author[T.~Kania]{Tomasz Kania\textsuperscript{*}}
\address[T.~Kania]{Mathematical Institute\\Czech Academy of Sciences\\\v Zitn\'a 25 \\115 67 Praha 1\\Czech Republic  and  Institute of Mathematics and Computer Science\\ Jagiellonian University\\ {\L}ojasiewicza 6, 30-348 Krak\'{o}w, Poland
}
\email{kania@math.cas.cz, tomasz.marcin.kania@gmail.com}
\thanks{\textsuperscript{*}Corresponding author: Tomasz Kania; \href{mailto:kania@math.cas.cz}{kania@math.cas.cz}.}
\thanks{RVO: 67985840.}

\date{}


\begin{document}

\begin{abstract}
Let $\cF=2^{\aleph_0}$.  We prove that $\ell_\infty$ contains
$2^{\cF}$ pairwise non-isomorphic non-reflexive Grothendieck subspaces.
They may all be chosen to contain the canonical copy of $c_0$ and to have an 
infinite-dimensional reflexive quotient.  This is optimal and answers
\cite[Problem~24]{GonzalezKania2021}.  We also show that the relative
$\ell_\infty$-Grothendieck property depends on the specified copy of $c_0$
and is not an isomorphic invariant of the subspace.  More generally, for
every closed $E\subseteq\ell_\infty$ there is an
$\ell_\infty$-Grothendieck subspace isomorphic to
$\ell_\infty\oplus_\infty E$; taking $E=c_0$ answers
\cite[Problem~1]{GonzalezLeonRomero2021}.
\end{abstract}

\maketitle

\section{Introduction}

A Banach space is \emph{Grothendieck} when weak-star convergent sequences in
its dual are weakly convergent.  Grothendieck proved that $\ell_\infty(\Gamma)$
has this property for every set $\Gamma$ \cite{Grothendieck1953}.  The
property is not inherited by closed subspaces, as the canonical copy of
$c_0$ in $\ell_\infty$ shows.  This makes the isomorphic diversity of
Grothendieck subspaces of $\ell_\infty$ a natural question.
Terminology and notation not explained in the Introduction are collected in
Section~\ref{sec:preliminaries}.

Problem~24 of \cite{GonzalezKania2021} asks whether the largest possible
number, namely $2^{\cF}$, can occur.  Our first result answers this question.

\begin{maintheorem}\label{thm:main}
There is a family $(M_\alpha)_{\alpha<2^{\cF}}$ of closed subspaces of
$\ell_\infty$ such that each $M_\alpha$ contains the canonical copy of
$c_0$, is non-reflexive and Grothendieck, and has infinite-dimensional
reflexive quotient $\ell_\infty/M_\alpha$.  The spaces $M_\alpha$ are
pairwise non-isomorphic.  Consequently, $\ell_\infty$ has exactly
$2^{\cF}$ isomorphism classes of Grothendieck subspaces.
\end{maintheorem}

The quotient result underlying Theorem~\ref{thm:main} is stated in terms of
Fredholm
operators.  Recall that an operator is Fredholm if its kernel and cokernel
are finite-dimensional and its range is closed.
Two Banach spaces are Fredholm equivalent when there is a Fredholm operator
between them, and Fredholm inequivalent otherwise.

\begin{maintheorem}\label{thm:quotients}
The space $\ell_\infty/c_0$ has $2^{\cF}$ infinite-dimensional reflexive
quotients which are pairwise Fredholm inequivalent.
\end{maintheorem}

We also consider the relative $\ell_\infty$-Grothendieck property introduced in
\cite{GonzalezLeonRomero2021}.  Its precise formulation, which involves the
specified copy of $c_0$, is recalled in Section~2.  It is not an isomorphic
property of the Banach space alone; this dependence on ambient position is
made explicit in Section~3.

\begin{maintheorem}\label{thm:relative-main}
For every closed subspace $E\subseteq\ell_\infty$, the space
$\ell_\infty\oplus_\infty E$ is isomorphic to an
$\ell_\infty$-Grothendieck subspace of $\ell_\infty$.  It is Grothendieck if
and only if $E$ is Grothendieck.  Consequently, an
$\ell_\infty$-Grothendieck subspace of $\ell_\infty$ need not be
Grothendieck.
\end{maintheorem}

To prove Theorem~\ref{thm:quotients}, let $1<s<a<b<2$ and let $A$ range over the
infinite subsets
of $[a,b]$.  The standard mixed sum $\ell_s(A,\ell_q)$ remembers $A$ through
the exponents $r$ for which it contains a complemented copy of $\ell_r$.
This information is invariant under Fredholm equivalence.  Two layers of stable
random variables embed $\ell_s(A,\ell_q)$ into an $L_1$-space, and a direct
measure-lifting argument places that $L_1$-space in
$c_0^\perp\subseteq\ell_\infty^*$.  Duality then produces the required
quotients of $\ell_\infty/c_0$.  Section~2 fixes the terminology and proves
the auxiliary transfer and embedding results needed later.  Section~3 proves
Theorem~\ref{thm:relative-main}, Section~4 proves
Theorem~\ref{thm:quotients}, and Section~5 completes the proof of
Theorem~\ref{thm:main}.

\section{Preliminaries}\label{sec:preliminaries}

All Banach spaces are over $\K\in\{\R,\C\}$, and all subspaces are closed.
We write $p'$ for the conjugate exponent of $1<p<\infty$.  For a family
$(E_i)_{i\in J}$, the standard notation $\ell_p(J,E_i)$ denotes its
$\ell_p$-sum.  In particular,
\[
 \ell_s(A,\ell_q)
 =\bigl\{(x_q)_{q\in A}:x_q\in\ell_q,
        \ (\|x_q\|_q)_{q\in A}\in\ell_s(A)\bigr\}.
\]
Its dual is canonically
$\ell_{s'}(A,\ell_{q'})$ whenever $1<s<\infty$ and
$A\subset(1,\infty)$.

We first formulate the relative $\ell_\infty$-Grothendieck property for a pair.
Let $S$ be a
Banach space and let $j:c_0\to S$ be an isomorphic embedding. By the 
Hahn--Banach theorem, $j^{**}:c_0^{**}=\ell_\infty\to S^{**}$ is an 
isomorphic embedding.

\begin{definition}\label{def:relative}
The pair $(S,j)$ is \emph{$\ell_\infty$-Grothendieck} if, whenever
$x_n^*\to x^*$ in $\sigma(S^*,S)$, one has
$\langle j^{**}z,x_n^*\rangle\to\langle j^{**}z,x^*\rangle$ for every
$z\in\ell_\infty$.

A closed subspace $S\subseteq\ell_\infty$ is an \emph{$\ell_\infty$-Grothendieck subspace} if $S$ contains the canonical $c_0$, and the pair $(S,j)$ is $\ell_\infty$-Grothendieck with $j$ 
equal to the inclusion.
\end{definition}

Obviously, being an $\ell_\infty$-Grothendieck subspace of $\ell_\infty$ is not invariant under isomorphisms. However, being an $\ell_\infty$-Grothendieck pair is invariant under isomorphisms in the following sense: if $U:S\to S_1$ is an isomorphism and $Uj=j_1V$ for an automorphism $V$ of $c_0$, then $(S,j)$ has the property if and only if $(S_1,j_1)$ does.  
Indeed, passing to biduals gives $U^{**}j^{**}=j_1^{**}V^{**}$, where $V^{**}$ is an
automorphism of $\ell_\infty$, and one pulls weak-star convergent sequences
back by $U^*$.  The latter property need not be
invariant under an isomorphism which does not respect the specified copy of
$c_0$; an explicit example is given after Theorem~\ref{thm:relative-construction}.
Every Grothendieck space $S$ makes $(S,j)$
$\ell_\infty$-Grothendieck for every embedding $j:c_0\to S$.

We shall also use Fredholm equivalence and two permanence principles.  Two
Banach spaces $E$ and $F$ are \emph{Fredholm equivalent} when there is a
Fredholm operator $E\to F$.  We shall use the standard equivalence
\[
 E\text{ and }F\text{ are Fredholm equivalent}
 \quad\Longleftrightarrow\quad
 E\oplus G\cong F\oplus H
\]
for some finite-dimensional spaces $G$ and $H$.  In particular, Fredholm
equivalence is symmetric.  To see the displayed equivalence, split the
kernel and range of a Fredholm operator; conversely, compose an isomorphism
$E\oplus G\to F\oplus H$ with the canonical inclusion of $E$ and the
projection onto $F$.

It also preserves the occurrence of a complemented copy of $\ell_r$,
$1<r<\infty$.  Indeed, suppose that $V\cong\ell_r$ is complemented in
$E\oplus G$, where $G$ is finite-dimensional.  Then
$V_0=V\cap E$ is the kernel of the finite-rank coordinate map $V\to G$, so
it has finite codimension in $V$ and is isomorphic to $\ell_r$.  If $P$ is
a projection onto $V$ and $R:V\to V_0$ is a projection, then
$R P|_E$ is a projection from $E$ onto $V_0$.  Thus adjoining or removing a
finite-dimensional summand does not change this property.

We require two permanence results.  If $Y$ is a subspace of a Grothendieck
space $X$ and $X/Y$ is either reflexive or separable, then $Y$ is
Grothendieck.  The reflexive case is
\cite[Corollary~1.4]{MartinezCervantesRodriguez2022}, and the separable case is
\cite[Proposition~3.1]{GonzalezLeonRomero2021}.
We also use the
following parts of the Lindenstrauss--Rosenthal theorem
\cite[Theorem~2.f.12]{LindenstraussTzafriri1996}.  If $M,N\subseteq
\ell_\infty$ and $U:M\to N$ is an isomorphism, then:
\begin{enumerate}
\item if $\ell_\infty/M$ and $\ell_\infty/N$ are non-reflexive, $U$ extends
to an automorphism of $\ell_\infty$;
\item if both quotients are infinite-dimensional and reflexive, $U$ has a
Fredholm extension to $\ell_\infty$.
\end{enumerate}
See also the original paper \cite{LindenstraussRosenthal1969}.

The following consequence transfers Theorem~\ref{thm:quotients} to kernels.

\begin{proposition}\label{prop:kernel-transfer}
Let $(Q_\gamma:\ell_\infty\ontoarrow X_\gamma)_{\gamma\in\Gamma}$ be
quotient maps onto infinite-dimensional reflexive spaces.  If the spaces
$X_\gamma$ are pairwise Fredholm inequivalent, then the kernels
$\ker Q_\gamma$ are pairwise non-isomorphic Grothendieck spaces.  If
$c_0\subseteq\ker Q_\gamma$ for every $\gamma$, all these kernels are
$\ell_\infty$-Grothendieck subspaces.
\end{proposition}

\begin{proof}
The first permanence result applies because $\ell_\infty$ is Grothendieck.
Suppose that $U:\ker Q_\gamma\to\ker Q_\delta$ is an isomorphism.  A
Fredholm extension $\widehat U$ supplied by Lindenstrauss--Rosenthal induces
an operator
\[
 \widetilde U:\ell_\infty/\ker Q_\gamma
 \longrightarrow\ell_\infty/\ker Q_\delta,
 \qquad \widetilde U(x+\ker Q_\gamma)
       =\widehat Ux+\ker Q_\delta.
\]
Its defect spaces are
\[
 \ker\widetilde U\cong
 \ker\widehat U/(\ker\widehat U\cap\ker Q_\gamma),
 \qquad
 \operatorname{coker}\widetilde U\cong
 \ell_\infty/(\ran\widehat U+\ker Q_\delta).
\]
The first is a quotient of the finite-dimensional space
$\ker\widehat U$.  The second is finite-dimensional, and
$\ran\widehat U+\ker Q_\delta$ is closed because it contains the closed
finite-codimensional space $\ran\widehat U$.  Hence $\widetilde U$ is
Fredholm.  Thus
$X_\gamma$ and $X_\delta$ are Fredholm equivalent, which forces
$\gamma=\delta$.  The last assertion follows from Definition~\ref{def:relative}
because the kernels are Grothendieck.
\end{proof}

We finally record the measure-theoretic embedding that will be used in
Section~4.  Write $\Nstar=\beta\N\setminus\N$ and identify
$C(\Nstar)$ with $\ell_\infty/c_0$.

\begin{proposition}\label{prop:L1-annihilator}
If $K$ is compact and separable and $\mu$ is a regular Borel probability
measure on $K$, then $L_1(K,\mu)$ embeds isometrically into
$c_0^\perp\subseteq\ell_\infty^*$.
\end{proposition}

\begin{proof}
Choose a sequence $(x_n)$ in $K$ whose range is dense and in which every
term occurs infinitely often.  Let $\tau(n)=x_n$ and let
$\beta\tau:\beta\N\to K$ be its continuous extension.  For $x\in K$ and a
neighbourhood $V$ of $x$, put $A_V=\{n:x_n\in V\}$.  Finite intersections
of these sets are infinite, so together with the cofinite sets they generate
a free filter.  Any ultrafilter extending it gives
$p\in\bigcap_V\overline{A_V}\cap\Nstar$, and
$\beta\tau(p)=x$ by regularity of $K$.  Hence
$\pi=\beta\tau|_{\Nstar}:\Nstar\to K$ is surjective.

The isometric unital embedding $f\mapsto f\circ\pi$ from $C(K)$ to
$C(\Nstar)$ allows the functional
$f\circ\pi\mapsto\int_K f\,d\mu$ to be extended, by Hahn--Banach, to a
norm-one functional on $C(\Nstar)$ taking $1$ to $1$.  This extension is
positive and hence is represented by a probability measure $\nu$ on
$\Nstar$ satisfying $\pi_*\nu=\mu$.

For $f\in C(K)$, the map $f\mapsto(f\circ\pi)\nu$ takes values in
$M(\Nstar)$ and satisfies
\[
 \|(f\circ\pi)\nu\|=\int_{\Nstar}|f\circ\pi|\,d\nu
 =\int_K|f|\,d\mu.
\]
Thus it extends by density to an isometry $L_1(K,\mu)\to M(\Nstar)$, and 
$M(\Nstar)=C(\Nstar)^*$ is the range of the adjoint of the
quotient map $\ell_\infty\to\ell_\infty/c_0$, hence isometric to
$c_0^\perp$.
\end{proof}

This construction may also be viewed through the split dual sequence
\[
 0\longrightarrow(\ell_\infty/c_0)^*
 \longrightarrow\ell_\infty^*\longrightarrow\ell_1\longrightarrow0.
\]
We use this only as a structural viewpoint: the splitting alone does not
place a given reflexive subspace in the first term, whereas the measure lift
above does so directly.

\section{The relative property and ambient position}

We first give an operator characterisation of Definition~\ref{def:relative}.

\begin{proposition}\label{prop:operator-criterion}
Let $j:c_0\to S$ be an isomorphic embedding.  The pair $(S,j)$ is
$\ell_\infty$-Grothendieck if and only if, for every bounded
$T:S\to c_0$, the operator $Tj:c_0\to c_0$ is weakly compact.
\end{proposition}

\begin{proof}
Suppose first that $(S,j)$ is $\ell_\infty$-Grothendieck.  If $(e_n^*)$ is
the canonical basis of $c_0^*=\ell_1$, then $(T^*e_n^*)$ is weak-star null
in $S^*$.  Consequently, for $z\in\ell_\infty$,
\[
 \langle (Tj)^{**}z,e_n^*\rangle
 =\langle j^{**}z,T^*e_n^*\rangle\longrightarrow0.
\]
Thus $(Tj)^{**}\ell_\infty\subseteq c_0$, and the bidual criterion for weak
compactness gives the conclusion.

Conversely, let $x_n^*\to x^*$ in $\sigma(S^*,S)$ and put
$u_n^*=x_n^*-x^*$.  Uniform boundedness and pointwise convergence show that
\(
 Ts=(\langle u_n^*,s\rangle)_{n\in\N}\qquad(s\in S)
\)
defines a bounded operator $T:S\to c_0$.  By assumption,
$(Tj)^{**}z\in c_0$ for every $z\in\ell_\infty$.  Since
$T^*e_n^*=u_n^*$, we obtain
$\langle j^{**}z,u_n^*\rangle=((Tj)^{**}z)_n\to0$.  This is precisely the
required convergence on $j^{**}\ell_\infty$.
\end{proof}

Let $\Phi:\ell_\infty\oplus_\infty\ell_\infty\to\ell_\infty$ be the
interlacing isometry
$\Phi(x,y)=(x_1,y_1,x_2,y_2,\ldots)$, and consider
$\Phi(c_0\oplus\{0\})$.  The map
$\theta(\Phi(x,0))=x$ is an onto isometry from this subspace to the
canonical $c_0$.  Moreover,
\[
 \ell_\infty/\Phi(c_0\oplus\{0\})
 \cong (\ell_\infty/c_0)\oplus_\infty\ell_\infty,
\]
whereas $\ell_\infty/c_0\cong C(\Nstar)$ is an infinite-dimensional
$C(K)$-space and hence is non-reflexive.  Thus both quotients are
non-reflexive.  Lindenstrauss--Rosenthal therefore provides an automorphism
$U$ of $\ell_\infty$ such that
\begin{equation}\label{eq:U-on-c0}
 U\Phi(x,0)=x\qquad(x\in c_0).
\end{equation}
The same automorphism can alternatively be obtained from the diagonal
principles for exact sequences in \cite[Section~2.11]{CabelloCastillo2023}.

The construction is summarised by the following commutative diagram with
exact rows; $\overline U$ is the isomorphism induced on the quotients.
{\normalsize
\[
\begin{tikzcd}[row sep=huge,column sep=huge]
0 \arrow[r] & \Phi(c_0\oplus\{0\}) \arrow[r] \arrow[d,"\theta"'] &
 \ell_\infty \arrow[r] \arrow[d,"U"'] &
 (\ell_\infty/c_0)\oplus_\infty\ell_\infty
 \arrow[r] \arrow[d,"\overline U"'] & 0\\
0 \arrow[r] & c_0 \arrow[r] & \ell_\infty \arrow[r] &
 \ell_\infty/c_0 \arrow[r] & 0.
\end{tikzcd}
\]
}

\begin{theorem}\label{thm:relative-construction}
For every closed $E\subseteq\ell_\infty$, the space
\[
 S_E=U\Phi(\ell_\infty\oplus_\infty E)
\]
is an $\ell_\infty$-Grothendieck subspace of $\ell_\infty$ and is isomorphic
to $\ell_\infty\oplus_\infty E$.  It is Grothendieck if and only if $E$ is
Grothendieck.
\end{theorem}

\begin{proof}
The space $S_E$ is closed and has the stated isomorphism type.  By
\eqref{eq:U-on-c0}, it contains the canonical $c_0$.  Given $T:S_E\to c_0$,
define $\widetilde T:\ell_\infty\to c_0$ by
$\widetilde T x=T(U\Phi(x,0))$.  Every operator from the Grothendieck space
$\ell_\infty$ to $c_0$ is weakly compact.  On $c_0$, identity
\eqref{eq:U-on-c0} gives $\widetilde T=T$, so Proposition~\ref{prop:operator-criterion}
shows that $S_E$ is $\ell_\infty$-Grothendieck.

The space $U\Phi(\{0\}\oplus E)$ is complemented in $S_E$ and isomorphic to
$E$.  Hence Grothendieckness of $S_E$ implies that of $E$.  Conversely,
$\ell_\infty\oplus_\infty E$ is Grothendieck whenever $E$ is, so the
isomorphic space $S_E$ is Grothendieck in this case.
\end{proof}

The theorem proves Theorem~\ref{thm:relative-main}, and $E=c_0$ gives the
example requested in
\cite[Problem~1]{GonzalezLeonRomero2021}.

\begin{example}\label{ex:ambient-position}
The relative notion is not intrinsic to the isomorphism class of the
subspace.  Put $R=\Phi(\ell_\infty\oplus_\infty c_0)$.  Both $R$ and
$S_{c_0}=UR$ contain the canonical $c_0$ and are isomorphic.  The second is
$\ell_\infty$-Grothendieck, whereas the first is not.  Indeed, define
$P:R\to c_0$ by $P\Phi(x,y)=y$.  Its restriction to the canonical $c_0$ is
the surjective even-coordinate projection $c_0\to c_0$, which is not weakly
compact, so Proposition~\ref{prop:operator-criterion} applies.  The mechanism
is that $U$ carries $\Phi(c_0\oplus\{0\})$, rather than the canonical $c_0$
contained in $R$, onto the canonical $c_0$.
\end{example}

\section{Fredholm-inequivalent reflexive quotients}

Let us fix
 $1<s<a<b<2$, $I=[a,b]$,
and let $\mathcal A$ be the family of all infinite subsets of $I$.  For
$A\in\mathcal A$, set
\begin{equation}\label{eq:XA-standard}
 X_A=\left(\bigoplus_{q\in A}\ell_{q'}\right)_{\ell_{s'}}
     =\ell_{s'}(A,\ell_{q'}).
\end{equation}
Thus $X_A^*=\ell_s(A,\ell_q)$ canonically, and both spaces are reflexive.
The restriction to infinite sets ensures that the outer exponent $s$ occurs
in the complemented-subspace structure below.  Since $|I|=\cF$, we also
have $|\mathcal A|=2^{\cF}$.

The relative positions of the outer exponent $s$ and the interval $I$, as
well as the encoding of a set $A\subseteq I$, are illustrated in
Figure~\ref{fig:parameters}.

\begin{figure}[H]
\centering
\begin{tikzpicture}[x=1.18cm,scale=1.05]
  \draw[->] (-0.5,0) -- (9.6,0);
  \foreach \x/\lab in {0/1,1.5/s,3/a,7.5/b,9/2}
    \draw (\x,0.12) -- (\x,-0.12) node[below] {$\lab$};
  \draw[line width=1.2pt] (3,0) -- (7.5,0);
  \foreach \x in {3.3,3.9,4.45,5.15,5.8,6.45,7.1}
    \fill (\x,0) circle (2.4pt);
  \node[above] at (5.25,0.25) {$A\subseteq I=[a,b]$};
\end{tikzpicture}
\caption{The outer exponent $s$ lies below $I$.  The complemented
$\ell_r$-exponents of $\ell_s(A,\ell_q)$ are $A\cup\{s\}$, so
$s\notin I$ allows one to recover $A$.}
\label{fig:parameters}
\end{figure}
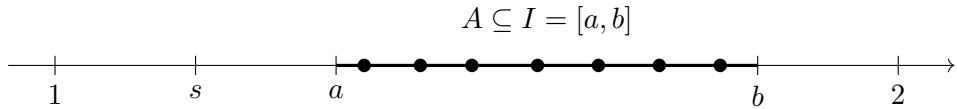

The following standard mixed-sum fact is the only structural information
about $\ell_p$-spaces that we need.

\begin{proposition}\label{prop:mixed-sum}
Let $A\in\mathcal A$ and $1<r<\infty$.  Then $\ell_r$ is isomorphic to a
complemented subspace of $\ell_s(A,\ell_q)$ if and only if
$r=s$ or $r\in A$.
\end{proposition}

\begin{proof}
Each coordinate copy of $\ell_q$, $q\in A$, is $1$-complemented.  Since
$A$ is infinite, choosing one unit vector in each of countably many
coordinates gives a $1$-complemented copy of $\ell_s$.

Conversely, let $J:\ell_r\to\ell_s(A,\ell_q)$ be an isomorphic embedding.
Every vector in an $\ell_s$-sum has countable support; applying this to a
countable dense subset of $J\ell_r$ shows that the range of $J$ lies in a
closed countable coordinate subsum.  Distinct $\ell_p$-spaces are totally
incomparable, as follows from the saturation result
\cite[Proposition~2.a.2]{LindenstraussTzafriri1996} and Pitt's theorem
\cite{Pitt1936}.

Let $P_q$ denote the coordinate projections.  If some $P_qJ$ is not
strictly singular, then $\ell_r$ and $\ell_q$ have isomorphic
infinite-dimensional subspaces, whence $r=q$.  Otherwise every finite
coordinate projection $P_FJ$ is strictly singular.  Enumerate the relevant
coordinates.  Inductively, after choosing $F_1,\ldots,F_{n-1}$, take a
normalised successive block $u_n$ so far out that
$\|P_{F_1\cup\cdots\cup F_{n-1}}Ju_n\|<\varepsilon_n$, and then choose a
finite set $F_n$, disjoint from its predecessors, which captures the
remaining coordinates of $Ju_n$ up to $\varepsilon_n$.  Thus
$\|Ju_n-P_{F_n}Ju_n\|<2\varepsilon_n$.  Choose the positive summable
sequence $(\varepsilon_n)$ small enough for the small-perturbation principle.
The
sequence $(u_n)$ is equivalent to the $\ell_r$-basis, whereas the disjointly
supported sequence $(P_{F_n}Ju_n)$ is equivalent to the $\ell_s$-basis.
Thus $r=s$.  This is the standard mixed-sum dichotomy; compare
\cite[Theorem~2.d.1]{LindenstraussTzafriri1996}.
\end{proof}

\begin{proposition}\label{prop:fredholm-separation}
If $A,B\in\mathcal A$ and a Fredholm operator $X_A\to X_B$ exists, then
$A=B$.
\end{proposition}

\begin{proof}
The adjoint is a Fredholm operator
$\ell_s(B,\ell_q)\to\ell_s(A,\ell_q)$.  Fredholm equivalence preserves the
exponents for which a complemented $\ell_r$ occurs.  Proposition~\ref{prop:mixed-sum}
therefore gives $B\cup\{s\}=A\cup\{s\}$.  Since $s\notin I$, we have
$A=B$.
\end{proof}

Proposition~\ref{prop:fredholm-separation} already gives the required
number of pairwise Fredholm-inequivalent
reflexive spaces.  We next realise all of them as quotients of
$\ell_\infty/c_0$.

We use stable random variables to realise these spaces as subspaces of
$L_1$-spaces.  For $0<r<2$, a real standard symmetric $r$-stable random
variable $Z_r$ is
normalised by
$\mathbb E\exp(itZ_r)=\exp(-|t|^r)$.  In the complex case we use an isotropic
complex $r$-stable variable, normalised by
\[
 \mathbb E\exp\bigl(i\operatorname{Re}(\overline z Z_r)\bigr)
 =\exp(-|z|^r)\qquad(z\in\C).
\]
If $Z_{r,j}$ are independent copies and $a_j$ are scalars, stability means
that $\sum_j a_jZ_{r,j}$ has the same distribution as
$(\sum_j|a_j|^r)^{1/r}Z_r$; the positive factor is its \emph{scale}.  For
$0<t<r$, put $m_{t,r}=\|Z_r\|_{L_t}$.  The moments are finite and depend
continuously on $(t,r)$ in this region.  These standard facts may be found
in \cite{BretagnolleDacunhaKrivine1967}.

\begin{theorem}\label{thm:stable-embedding}
There are constants $0<c\leqslant C<\infty$, depending only on $s,a,b$,
such that, for every $A\in\mathcal A$, the space
$\ell_s(A,\ell_q)$ embeds into $L_1(K_A,\mu_A)$ with distortion at most
$C/c$, where
\[
 K_A=[0,1]^{\Gamma_A},\qquad
 \Gamma_A=(A\times\N)\mathbin{\dot\cup}A,
\]
and $\mu_A$ is product Lebesgue measure.  The compact space $K_A$ is
separable.
\end{theorem}

\begin{proof}
On the indicated product space choose mutually independent standard
$q$-stable variables $\zeta_{q,n}$, $(q,n)\in A\times\N$, and mutually
independent standard $s$-stable variables $\theta_q$, $q\in A$, the two
families being independent.  They are realised on the cube by applying
Borel quantile maps to the corresponding coordinate functions.  For a
finitely supported
$x=(x_q)_{q\in A}$ with finitely supported coordinates, define
\[
 \Psi_Ax=\sum_{q\in A}\theta_q
             \sum_{n=1}^{\infty}x_q(n)\zeta_{q,n}.
\]
All sums are finite at this stage.  Conditioning first on the inner sums
identifies the conditional $L_1$-norm through $s$-stability.  Jensen's
inequality gives the lower estimate and Lyapunov's inequality gives the
upper estimate:
\[
 m_{1,s}\Big(\sum_{q\in A}m_{1,q}^s\|x_q\|_q^s\Big)^{1/s}
 \leqslant \|\Psi_Ax\|_{L_1}
 \leqslant
 m_{1,s}\Big(\sum_{q\in A}m_{s,q}^s\|x_q\|_q^s\Big)^{1/s}.
\]
The assumption $s<a\leqslant q$ ensures that the $s$th moments exist.
Continuity and positivity of the moments on the compact interval $[a,b]$
give uniform lower and upper constants.  Density extends $\Psi_A$ to the
required embedding.

Finally, $|\Gamma_A|\leqslant\cF$.  The Hewitt--Marczewski--Pondiczery
theorem \cite[Theorem~2.3.15]{Engelking1989} shows that a product of at most
$\cF$ separable spaces is separable, so $K_A$ is separable.
\end{proof}

If $A$ is countable, the probability space in the theorem may be replaced
by the standard interval and the target by $L_1(0,1)$.  This cannot be done
for uncountable $A$: the space $\ell_s(A,\ell_q)$ has uncountable density,
whereas $L_1(0,1)$ is separable.  The topological separability of $K_A$, not
norm separability of $L_1(K_A,\mu_A)$, is what is used below.

Combining the stable-variable embedding with the preliminary embedding into
$c_0^\perp$ yields the required quotient maps.
\begin{proposition}\label{prop:quotient-map}
For every $A\in\mathcal A$ there is a quotient map
$Q_A:\ell_\infty\ontoarrow X_A$ such that $c_0\subseteq\ker Q_A$.
\end{proposition}

\begin{proof}
Theorem~\ref{thm:stable-embedding} and Proposition~\ref{prop:L1-annihilator}
give an embedding $\Psi_A:\ell_s(A,\ell_q)\to L_1(K_A,\mu_A)$ and an
isometry $R_A:L_1(K_A,\mu_A)\to M(\Nstar)$.  Let
$q:\ell_\infty\to\ell_\infty/c_0$ be the quotient map.  Its adjoint
identifies $M(\Nstar)$ with $c_0^\perp$; put
$J_A=q^*R_A\Psi_A$.  Define
$Q_A:\ell_\infty\to X_A=\ell_s(A,\ell_q)^*$ by
\[
 \langle Q_Ax,y\rangle=\langle J_Ay,x\rangle
 \qquad(x\in\ell_\infty,\ y\in\ell_s(A,\ell_q)).
\]
Reflexivity gives $Q_A^*=J_A$.  Since $J_A$ is bounded below, the
closed-range theorem shows that $\ran Q_A$ is closed.  Moreover,
$(\ran Q_A)^\perp=\ker Q_A^*=\ker J_A=\{0\}$, so the range is dense and
hence $Q_A$ is surjective.  Its vanishing on $c_0$
follows from $J_A(\ell_s(A,\ell_q))\subseteq c_0^\perp$.
\end{proof}

Since $Q_A$ annihilates $c_0$, we have
$Q_A=\overline Q_Aq$ for a quotient map $\overline Q_A$.  The two halves of
the realisation are displayed below.
{\large
\[
\begin{tikzcd}[row sep=huge,column sep=large]
\ell_s(A,\ell_q) \arrow[r,"\Psi_A"]
 \arrow[drr,bend right=18,"J_A"'] &
L_1(K_A,\mu_A) \arrow[r,"R_A"] &
M(\Nstar) \arrow[d,"q^*","\cong"']\\
&& c_0^\perp \arrow[r,hook] & \ell_\infty^*
\end{tikzcd}
\]
\vspace{-0.4em}
\[
\begin{tikzcd}[row sep=huge,column sep=huge]
\ell_\infty \arrow[r,two heads,"q"] \arrow[dr,two heads,"Q_A"'] &
\ell_\infty/c_0 \arrow[d,two heads,"\overline Q_A"]\\
& X_A.
\end{tikzcd}
\]
}

\begin{proof}[Proof of Theorem~\ref{thm:quotients}]
The family $\mathcal A$ has cardinality $2^{\cF}$.  The targets of the
quotient maps in Proposition~\ref{prop:quotient-map} are reflexive and
infinite-dimensional, and Proposition~\ref{prop:fredholm-separation} shows
that they are pairwise Fredholm inequivalent.  Since every $Q_A$ annihilates
$c_0$, it factors through $\ell_\infty/c_0$.
\end{proof}

\section{Proof of the main theorem}

For every $A\in\mathcal A$, the exact sequences used below fit into the
commutative diagram
{\normalsize
\[
\begin{tikzcd}[row sep=huge,column sep=huge]
0 \arrow[r] & \ker Q_A \arrow[r] \arrow[d,two heads] &
\ell_\infty \arrow[r,"Q_A"] \arrow[d,two heads,"q"'] &
X_A \arrow[r] \arrow[d,equal] & 0\\
0 \arrow[r] & (\ker Q_A)/c_0 \arrow[r] &
\ell_\infty/c_0 \arrow[r,"\overline Q_A"'] & X_A \arrow[r] & 0.
\end{tikzcd}
\]
}

\begin{proof}[Proof of Theorem~\ref{thm:main}]
Apply Proposition~\ref{prop:kernel-transfer} to the maps
$(Q_A)_{A\in\mathcal A}$.  Their kernels are pairwise non-isomorphic
Grothendieck subspaces of $\ell_\infty$, they contain the canonical $c_0$,
and their quotients are the infinite-dimensional reflexive spaces $X_A$.
No kernel can be reflexive, since reflexivity is a three-space property and
$\ell_\infty$ is not reflexive.  This gives $2^{\cF}$ isomorphism classes.
The reverse inequality follows from $|\ell_\infty|=\cF$, because
$\ell_\infty$ has at most $2^{\cF}$ subsets.
\end{proof}

\begin{corollary}
There are $2^{\cF}$ pairwise non-isomorphic
$\ell_\infty$-Grothendieck subspaces of $\ell_\infty$.
\end{corollary}

\section*{Acknowledgements}
The authors are grateful to the anonymous referee for a careful reading and
for detailed suggestions which substantially improved both the proofs and
the presentation.  In particular, the referee's question led us to isolate
Example~\ref{ex:ambient-position}.

\section*{Conflict of interest statement}
The authors declare no conflict of interest.

\section*{Data availability statement}
Data sharing is not applicable to this article as no datasets were generated
or analysed during the current study.

\end{document}